\documentclass{aic}

\usepackage[ruled,linesnumbered,vlined]{algorithm2e}
\usepackage{tikz}
\usepackage{enumitem}

\def\red{\text{\rm red}}
\newcommand{\sr}{\hat{r}}
\let\eps=\varepsilon
\def\ext{{\rm ext}}

\newtheorem{theorem}{Theorem}
\newtheorem{lemma}[theorem]{Lemma}
\newtheorem{corollary}[theorem]{Corollary}

\newtheorem{claim}[theorem]{Claim}

\newtheorem{definition}[theorem]{Definition}
\newtheorem{fact}[theorem]{Fact}

\aicAUTHORdetails{%
  title = {The Size-Ramsey Number of $3$-uniform Tight Paths}, 
  author = {Jie Han, Yoshiharu Kohayakawa, Shoham Letzter, Guilherme Oliveira Mota and Olaf Parczyk},
  plaintextauthor = {Jie Han, Yoshiharu Kohayakawa, Shoham Letzter,
    Guilherme Oliveira Mota and Olaf Parczyk},
  plaintexttitle = {The Size-Ramsey Number of 3-uniform Tight Paths}, 
  runningauthor = {J. Han, Y. Kohayakawa, S. Letzter, G. O. Mota and O. Parczyk},
  keywords = {size-Ramsey number, hypergraph, tight path},
}   

\aicEDITORdetails{%
   year={2021},
   number={5},
   received={23 June 2020},   
   published={11 June 2021},  
   doi={10.19086/aic.24581},      
}   

\begin{document}

\begin{frontmatter}[classification=text]

\title{The Size-Ramsey Number of $3$-uniform Tight Paths} 

\author[jie]{Jie Han\thanks{Partially supported by the Simons Collaboration Grant for Mathematicians \#630884. }}
\author[yoshi]{Yoshiharu Kohayakawa\thanks{Partially supported by CNPq (311412/2018-1, 423833/2018-9) and FAPESP (2018/04876-1, 2019/13364-7)}}
\author[shoham]{Shoham Letzter\thanks{Research supported by the Royal Society.}}
\author[gui]{Guilherme Oliveira Mota\thanks{Partially supported by CNPq (304733/2017-2, 428385/2018-4) and FAPESP (2018/04876-1, 2019/13364-7).}}
\author[olaf]{Olaf Parczyk\thanks{Partially supported by Technische Universit\"at Ilmenau, the Carl Zeiss Foundation, and the DFG (Grant PA 3513/1-1).\newline
		The cooperation of the authors was supported by a joint CAPES/DAAD
		PROBRAL project (Proj.~430/15, 57350402, 57391197).
		This study was financed in part by CAPES, Coordena\c{c}\~ao de
		Aperfei\c{c}oamento de Pessoal de N\'ivel Superior, Brazil, Finance
		Code~001.
		FAPESP is the S\~ao Paulo Research Foundation.  CNPq is the National
		Council for Scientific and Technological Development of
		Brazil.}}

\begin{abstract}
Given a hypergraph $H$, the size-Ramsey number $\hat{r}_2(H)$ is the smallest
integer $m$ such that there exists a hypergraph $G$ with $m$ edges with the property
that in any colouring of the edges of $G$ with two colours there is a
monochromatic copy of $H$.
We prove that the size-Ramsey number of the
$3$-uniform tight path on $n$ vertices $P^{(3)}_n$ is linear in $n$,
i.e., $\hat{r}_2(P^{(3)}_n) = O(n)$.
This answers a question by Dudek, La Fleur, Mubayi, and R\"odl for
$3$-uniform hypergraphs~[\emph{On the 
	size-Ramsey number of hypergraphs}, J.~Graph Theory \textbf{86} (2016),
417--434], who proved $\hat{r}_2(P^{(3)}_n) = O(n^{3/2} \log^{3/2} n)$.
\end{abstract}
\end{frontmatter}

\section{Introduction}

For hypergraphs $G$ and $H$ and an integer $s$, we denote by $G \rightarrow
(H)_s$ the property that in any $s$-colouring of the edges of $G$ there is a
monochromatic copy of $H$.
The \emph{$s$-colour} size-Ramsey number $\hat{r}_s(H)$ is
\begin{align*}
\sr_s(H) := \min \{ |E(G)| : G \rightarrow (H)_s \}.
\end{align*}

For the $n$-vertex path $P_n$, Erd\H{o}s~\cite{erdHos1981combinatorial} asked if $\sr_2(P_n)=O(n)$, which was
answered positively by Beck~\cite{Be83} using the probabilistic method.
An explicit construction for the same results was given by Alon and
Chung~\cite{AlCh88}.
Many successive improvements led to the currently best known bounds $3.75n -
o(n)\leq\sr_2(P_n)\leq 74n$ (see, e.g.,
\cite{Be83, bollobas1986extremal, dudek2017some,BalDeBiasio} for lower
bounds, and \cite{Be83, DuPr15, letzter16:_path_ramsey, dudek2017some} for
upper bounds).
For $s\geq 2$ colours, Dudek and Pra{\l}at~\cite{dudek2017some} and
Krivelevich~\cite{krivelevich2017long} proved that there are constants $c$ and
$C$ such that $cs^2n\leq\sr_s(P_n)\leq Cs^2(\log s)n$.

The systematic investigation of size-Ramsey questions for hypergraphs was
initiated by Dudek, La Fleur, Mubayi, and R\"odl~\cite{DFMR_sizeRamsey}.
Besides cliques and trees, they studied generalisations of paths.

We say that an $r$-uniform hypergraph is an \emph{$\ell$-path} if
there exists an ordering of its vertices such that every edge is
composed of $r$ consecutive vertices, two (vertex-wise) consecutive
edges share exactly $\ell$ vertices, and every vertex is contained in
an edge. 
For $1 \le \ell \le r-1$, let $P^{(r)}_{n,\ell}$ denote the \emph{$r$-uniform $\ell$-path} on $n$
vertices and for the \emph{tight path}, where $\ell=r-1$, we write $P^{(r)}_{n}$.
Dudek, La Fleur, Mubayi, and R\"odl~\cite{DFMR_sizeRamsey} deduced from Beck's result~\cite{Be83} that $\hat{r}_2(P^{(r)}_{n,\ell}) = O(n)$, when $1 \le
\ell \le r/2$.
Furthermore, they proved that $\hat{r}_2(P^{(r)}_{n})
= O_r(n^{r-1-\alpha} \log^{1+ \alpha}n)$ with $\alpha=(r-2)/(\binom{r-1}{2}+1)$, which
gives $\hat{r}_2(P^{(3)}_{n}) = O(n^{3/2} \log^{3/2}n)$.

This was improved and extended to more colours by Lu and
Wang~\cite{lu2018size}, who showed that $\hat{r}_s(P^{(r)}_{n}) = O_r(s^r (n \log
n)^{r/2})$ for $s \ge 2$ colours.
Dudek, La Fleur, Mubayi, and R\"odl~\cite{DFMR_sizeRamsey} asked if $\hat{r}_2(P^{(r)}_n)=O_r(n)$ for $r \ge 3$.
We answer this question for $3$-uniform hypergraphs by proving the following result.

\begin{theorem}
	\label{thm:main}
	The $2$-colour size-Ramsey number of the $3$-uniform tight path is
	\begin{align*}
	\sr_2(P^{(3)}_n) = O(n).	
	\end{align*}
\end{theorem}

Trivially, we need at least $n$ edges, so this is asymptotically optimal.
As observed in~\cite{DFMR_sizeRamsey}, bounds on size-Ramsey numbers 
for some uniformity can be used to obtain bounds for larger uniformities.
We obtain the following corollary.

\begin{corollary}
	For any integer $r$ such that $3\mid r$, the $2$-colour size-Ramsey
	number of the $r$-uniform $(2r/3)$-path is
	\begin{align*}
	\sr_2(P^{(r)}_{n,2r/3}) = O(n).
	\end{align*} 
\end{corollary}

To see this, take the graph given by 
Theorem~\ref{thm:main} and replace every vertex by 
a set of $r/3$ vertices. Then each $3$-edge naturally 
gives an $r$-edge, and every $3$-uniform tight path 
becomes an $r$-uniform $(2r/3)$-path.

Our proof combines new ideas and the method developed by Clemens, Jenssen, Kohayakawa,
Morrison, Mota, Reding, and Roberts~\cite{2colourSizeRamsey} for estimating the 
size-Ramsey number of powers of paths (see also~\cite{BKMMMMP,CYKMMRB}).
It is plausible that ideas from~\cite{CYKMMRB, BKMMMMP} may provide a strategy to solve the case with $s \ge 3$ colours.
However, the question whether the size-Ramsey number of a tight path is linear for hypergraphs with uniformity $r\ge 4$ remains open and requires additional ideas.

\section{Preliminaries}
\label{sec:prelims}

In this short section, we give a sketch of our proof of
Theorem~\ref{thm:main} and state two simple lemmas about random graphs. 

\subsection{Sketch of the proof of Theorem~\ref{thm:main}}

We will first sketch a proof for $\sr_2(P_n) = O(n)$.  It is not hard
(cf.~Lemmas~\ref{lem:Hexists} and~\ref{lem:H2} below) to obtain a
graph $G$ with $O(n)$ edges such that for any two sufficiently large
and disjoint sets of vertices $A$ and $B$ there is a path of
length~$n$ alternating between $A$ and $B$.  Given such a graph $G$,
we show that 
$G \rightarrow (P_n)_2$.  Consider an arbitrary $2$-colouring of the
edges of $G$ with colours blue and red.  
If there is no blue $P_n$ in $G$ we can show 
(cf.~Lemma~\ref{lem:transversal} below) that there are two
sets $A$ and $B$ of size at least $n$ with no blue edges in between.
By the property of~$G$ mentioned above there exists a $P_n$
alternating between $A$ and $B$, which unequivocally has to be red.


For the proof of Theorem~\ref{thm:main} we follow, in principle, the
same strategy.  Based on a blow-up of a power of a similar graph $G$,
we define a $3$-uniform hypergraph $H$ and claim that
$H \rightarrow (P^{(3)}_{n})_2$.  
We define an auxiliary (generalised)
graph $F$ on $V(G)$, which has $2$- and $3$-edges, such that a long 
path in $F$ gives a blue $P^{(3)}_n$ in $H$.
If~$F$ does not contain a long path, then we find a family of disjoint
sets such that no edge of $F$ lies between these sets
(cf.~Lemma~\ref{lem:transversal}).  Then by the properties of $G$
there exists a path in $G$ alternating through these sets.  As there
are no edges of $F$ `interfering' with this path, we are able to turn
it into a red $P^{(3)}_n$ in~$H$.

In the next section we provide the lemmas needed to obtain $G$.
Afterwards, in Section~\ref{23-graphs} we introduce the notion of
$(2,3)$-graphs, which, as can be seen above, plays a key role in
our argument.  Finally, we prove Theorem~\ref{thm:main} in
Section~\ref{sec:proof}.

\subsection{Sparse graphs with many long paths}
\label{sec:graph_G}

The following two lemmas are proved in~\cite{2colourSizeRamsey}.
Basically, together they imply that for every $k$ and $n$ there exists
a graph $G$ with $O_k(n)$ edges such that, for any disjoint sets of
vertices $A_1,\dots,A_{k+1}$ that are large enough, there exists a
path of length~$n$ `alternating' through these sets.

\begin{lemma}[{\cite[{Lemma 3.1}]{2colourSizeRamsey}}]
	\label{lem:Hexists}
	For every pair of positive constants $\varepsilon$ and $a$, there is a
	constant~$b$ such that, for any large enough~$n$, there is a graph
	$H$ with $v(H)=an$ and $\Delta(H)\leq b$ such that the following holds:
	\begin{enumerate}[label=\upshape(P\arabic*$_n$)]
		\item For every pair of disjoint sets $S, T\subseteq V(H)$ with $|S|, |T|\geq\varepsilon n$, we get $e_H(S,T)>0$\label{prop:exist}.
	\end{enumerate}
\end{lemma}

\begin{lemma}[{\cite[{Lemma 3.5}]{2colourSizeRamsey}}]
	\label{lem:H2}
	For every integer $k\geq 1$ and every $\eps>0$ there exists an
	integer $a$ such that the following holds.  Let $H$ be a graph on at
	least $am$ vertices such that for every pair of disjoint sets $S$,
	$T\subseteq V(H)$ with $|S|$, $|T|\geq \eps m$ we have
	$e_{H}(S,T)> 0$.  Then the following holds:
	\begin{enumerate}[label=\upshape(P\arabic*$_m$)]
		\setcounter{enumi}{1}
		\item For every family $A_1,\dots,A_{k+1}\subseteq V(H)$ of pairwise
		disjoint sets each of size at least $\eps a m$, there is a path
		$P_m=(x_1,\dots,x_m)$ in~$H$ with $x_i \in A_j$ for
		all~$1\leq i\leq m$, where $j\equiv i\pmod{k+1}$.\label{prop:path} 
	\end{enumerate}
\end{lemma}

Note that the hypothesis on~$H$ in Lemma~\ref{lem:H2}
is~\hyperref[prop:exist]{\upshape(P1$_m$)} from Lemma~\ref{lem:Hexists}.  Therefore, roughly
speaking, Lemma~\ref{lem:H2} tells us that~\hyperref[prop:exist]{\upshape(P1$_m$)}
implies~\ref{prop:path}.

\section{\texorpdfstring{$(2,3)$}{(2,3)}-graphs}
\label{23-graphs}

In this section we introduce a structure that helps us to transfer
some ideas from the graph case to the hypergraphs setting.  A
\emph{$(2,3)$-graph} $F=(V,E)$ consists of a set of vertices $V$ and a
set $E$ of \emph{$2$-edges} of the form $\{u,v\}$ and \emph{$3$-edges} of
the form $(\{u,v\} ,w)$, for distinct vertices $u,v,w\in V$. For simplicity
we will write~$uv$ for~$\{u,v\}$ and~$uv(w)$  for~$(\{u,v\}, w)$.
A sequence of vertices $P =(x_1, \dots, x_m)$ is a $(2,3)$-path of length $m$ in $F$ if for
every $i=1,\dots,m-1$ either $x_ix_{i+1}\in E$ or
$x_ix_{i+1}(w_i)\in E$ for some~$w_i\in V\setminus\{x_1,\dots,x_m\}$,
with all the~$w_i$ distinct.

Given pairwise disjoint sets $V_1,\dots,V_{k+1}$, we say that an edge
$uv \in E(F)$ ($uv(w) \in E(F)$) is a \emph{transversal with respect to
	$V_1,\dots,V_{k+1}$}, if $u$ and $v$ ($u$, $v$, and $w$) are in different sets $V_i$.
When the sets $V_1,\dots,V_{k+1}$ are clear from the context we say that the edge is a
\emph{transversal}.

We want to prove that if a sufficiently large $(2,3)$-graph $F=(V,E)$ contains no $(2,3)$-path with $n$ vertices, then there exist large disjoint sets $V_1,\dots,V_{k} \subseteq V$  such that $E$ contains no transversals and that there is no edge $uv(w)$ with $u \in V_1 \cup \dots \cup V_{k-1}$ and $v,w \in V_{k}$.
The last property is only required to support our inductive proof.
To prove this we use a \emph{Depth First
	Search (DFS)} algorithm.  For example, Ben-Eliezer, Krivelevich, and
Sudakov~\cite{ben-eliezer12:_Ramsey} used a DFS algorithm to find long
paths in expanding graphs to obtain bounds on the size-Ramsey number
of directed paths.  Their algorithm traverses the vertices of the input
graph and maintains a set $S$ of vertices that are fully dealt with,
a set $U$ of currently active vertices, and a set $J$ of vertices that
were not considered so far.  The set $U$ always spans a path and in
every step, if at all possible, this path is extended by adding a vertex from $J$.
Otherwise, the last vertex of the path is removed and added to $S$.
It is immediate that there cannot be any edges between $S$ and $J$ and if $U$ stays small, then at 
some point during the execution both $S$ and $J$ are large.

We adapt this algorithm to the setting of $(2,3)$-graphs (see Algorithm~\ref{algo:DFS} below).
As in the graph case we greedily extend a $(2,3)$-path (preferring $2$-edges over $3$-edges)
and backtrack if it gets stuck. We will now give the details of our algorithm. The input is a
$(2,3)$-graph $F=(V,E)$, disjoint subsets of vertices $V_1,\dots,V_k$, and an ordering of the
vertices $V=\{ v_1,\dots,v_N \}$.  During the algorithm we maintain
sets $S$, $T$, $W_S$, $W_U$,~$T_i$ for $i\in[k]$ and a $(2,3)$-path~$U$ as follows:
\begin{itemize}
	\item $S \subseteq V'$ is the set of vertices that are fully dealt
	with. 
	\item $W_S \subseteq V$ is the set of vertices $w$ that were `used' by
	vertices from $S$. 
	\item $U$ contains the currently active vertices in a $(2,3)$-path. 
	\item $W_U \subseteq V$ is the set of vertices $w$ that are `used' by
	the path $U$.
	\item $T_1\cup\dots\cup T_k$ are disjoint and $T_i\subset V_i$ for $i\in[k]$.
\end{itemize}
In every step of the algorithm, either the $(2,3)$-path $U$ is
extended by adding a vertex from $T_k$ to it or this is not possible,
and the last vertex from $U$ is removed and put into $S$.  While the
algorithm runs, after each execution of the while loop, we have the
following invariants, where $m$ is the length of the $(2,3)$-path
$U$:
\begin{enumerate}[label=\upshape(A\arabic*)]
	\item \label{prop:Upath} $U=(u_1,\dots,u_{m})$ is a $(2,3)$-path and
	$W_U$ is the set of the vertices~$w$ in the edges~$u_iu_{i+1}(w)$
	($1\leq i<m$) in the $(2,3)$-path $U$.

	\item \label{prop:Vcovered} $S$, $U\subseteq V_k$, $W_U\subseteq T_1\cup\dots\cup T_{k-1}$, $T_i\subseteq V_i$ for $i\in[k]$, $|W_S| \le |S|$, and $|W_U| \le \max\{ 0, m-1 \}$. 
\end{enumerate}

This process is described in Algorithm~\ref{algo:DFS}.

\IncMargin{1em}
\begin{algorithm}
	\SetKwData{Left}{left}\SetKwData{This}{this}\SetKwData{Up}{up}\SetKwRepeat{Do}{do}{while}
	\SetKwFunction{Union}{Union}\SetKwFunction{FindCompress}{FindCompress}
	\SetKwInOut{Input}{Input}
	\Input{A $(2,3)$-graph $F=(V,E)$, disjoint subsets of vertices $V_1,\dots,V_k$, and an ordering of the vertices $V=\{ v_1,\dots,v_N \}$.}
	\BlankLine
	Define $m \leftarrow 0$, $S\leftarrow \emptyset$, $W_S \leftarrow \emptyset$, $W_U \leftarrow \emptyset$, $T_i\leftarrow V_i$ for $1\leq i\leq k$\label{algo:init}\;
	\While{$T_k\neq\emptyset$\label{algo:while}}{
		\eIf{$m=0$}{
			Let $v$ be the vertex with smallest index from $T_k$\;
			$u_1 \leftarrow v$, $m\leftarrow 1$, $T_k \leftarrow T_k \setminus \{v\}$\label{algo:m0}\;
		}{
			Let $T_\ext \leftarrow \{ v\in T_k \colon u_mv \in E \text{ or } u_mv(w) \in E \text{ with }w \in T_1\cup\cdots\cup T_k \}$\label{algo:Tp}\; 
			\eIf{$T_\ext\not= \emptyset$\label{algo:Tempty}}{
				Let $v $ be the vertex with the smallest index from $T_\ext$\;
				$u_{m+1} \leftarrow v$, $T_k \leftarrow T_k \setminus \{v\}$\label{algo:m+1}\;
				\If{$u_mu_{m+1} \not \in E$}
				{
					Let $w \in T_1\cup\cdots\cup T_k $ be the vertex of smallest index such that $u_mu_{m+1}(w) \in E$;
					\tcp{There is one because $u_{m+1} \in T_\ext$.}
					$W_U \leftarrow W_U \cup \{w\}$ and $T_i \leftarrow T_i \setminus \{w\}$, where $w\in T_i$\label{algo:WU+}\;
				}
				$m \leftarrow m+1$\;
			}{
				$S \leftarrow S \cup \{ u_m \}$\label{algo:m-1}\;
				\If{$m>1$}
				{
				\If{$u_{m-1}u_m \not \in E$} 
				{
					Let $u_{m-1}u_m(w) \in E$ with $w \in W_U$ be the edge used by the $(2,3)$-path; \tcp{This is well defined by~\ref{prop:Upath}.}
					$W_S \leftarrow W_S \cup \{w\}$ and $W_U \leftarrow W_U \setminus \{w\}$ \label{algo:WU-}\;
				}
			}
				$m \leftarrow m-1$\;
			}
		}
	}
	\caption{DFS algorithm for traversing a $(2,3)$-graph.} \label{algo:DFS}
\end{algorithm}\DecMargin{1em}

\begin{lemma}
	\label{lem:algo}
	Algorithm~\ref{algo:DFS} terminates and Properties~\ref{prop:Upath} and~\ref{prop:Vcovered} hold throughout.
\end{lemma}

\begin{proof}[Proof of Lemma~\ref{lem:algo}]
	Observe that~\ref{prop:Upath} and~\ref{prop:Vcovered} hold when we
	initialise the sets and put $m=0$ on line~\ref{algo:init}.  Assume
	that we are in some step of the algorithm,
	where~\ref{prop:Upath} and~\ref{prop:Vcovered} hold and we have
	vertices $U=(u_1,\ldots,u_m)$ forming a $(2,3)$-path (this is true
	because of~\ref{prop:Upath}).
	
	Now we consider the next execution of the while loop.
	We know from~\ref{prop:Upath} that~$W_U$ contains exactly the vertices used
	in the edges $u_iu_{i+1}(w)$ for $i=1,\dots,m-1$.  Either we extend
	the path by an edge $u_mv$ or $u_mv(w)$ (lines~\ref{algo:Tp}
	and~\ref{algo:m+1}) where $w$ is added to $W_U$ if needed
	(line~\ref{algo:WU+}), or we remove an edge $u_{m-1}u_m$ or
	$u_{m-1}u_m(w)$ (lines~\ref{algo:m-1} and~\ref{algo:WU-}) where $w$
	is removed from $W_U$ if needed (line~\ref{algo:WU-}).
	Therefore,~\ref{prop:Upath} still holds.

	For~\ref{prop:Vcovered} it is easy to see that $T_i\subseteq V_i$ for $i\in[k]$, as in the beginning of the execution we have $T_i=V_i$ and no vertex is added to $T_i$.
	Also, since every $w$ in $W_U$ comes from $T_1\cup\dots\cup T_k$, we have $W_U\subseteq T_1\cup\dots\cup T_{k-1}$ (see lines~\ref{algo:Tp} and~\ref{algo:WU-}).
	Since every vertex of $S$ comes from $U$ (line~\ref{algo:m-1}) and every vertex of $U$ comes from $T_k\subset V_k$ (lines~\ref{algo:m0},~\ref{algo:Tp} and~\ref{algo:m+1}), which implies that $S$, $U\subset V_k$.  
	To prove that $|W_S| \le |S|$, it is enough to
	observe that line~\ref{algo:WU-} can only be executed after an
	execution of line~\ref{algo:m-1}.  Similarly, we have
	$|W_U| \le m-1$ with $m \ge 2$, because line~\ref{algo:WU+} can only
	be executed after an execution of line~\ref{algo:m+1}, and $|W_U|=0$
	with $m=1$, because on line~\ref{algo:m0} nothing is added to $W_U$.
	Thus,~\ref{prop:Vcovered} also remains true.
	
	It remains to show that the algorithm terminates.
	In every execution of the while loop, either one vertex from $T_k\subseteq V_k$
	is added to the path $U$ (lines~\ref{algo:m0} and~\ref{algo:m+1}) or
	moved from~$U$ to~$S$ (line~\ref{algo:m-1}).
	Therefore, after at most $2|V_k|$ steps we have $T=\emptyset$, and the
	algorithm terminates.
\end{proof}

We are ready to prove the aforementioned result on $(2,3)$-graphs $F$ with no long $(2,3)$-paths.

\begin{lemma}
	\label{lem:transversal}
	Let $k$, $c$ and $n$ be positive integers and let $F=(V,E)$ be a $(2,3)$-graph on at least $5^{k-1}cn$ vertices.
	If $F$ contains no $(2,3)$-path with $n$ vertices, then there exist disjoint sets
	$V_1,\dots,V_{k} \subseteq V$ of size at least $cn$ such that no
	edge from $E$ is a transversal and there is no edge $uv(w)$ with $u \in V_1 \cup \dots \cup V_{k-1}$ and $v,w \in V_{k}$.
\end{lemma}
\begin{proof}
	We prove the result by induction on $k$.
	For $k=1$ the result follows by putting $V_1=V(F)$.
	Thus let $k\geq 1$ and assume the statement holds for $k$.
	
	To prove the result for $k+1$, let $F$ be a $(2,3)$-graph on $5^k cn$ vertices which does not have a path of length $n$.
	In particular, $F$ does not have a path of length $5n$, and by the assumption on $k$, there exist disjoint sets $V_1,\cdots,V_k$, each of size $5cn$, such that no edge is a transversal, and there is no edge $uv(w)$ with $u \in V_1 \cup \dots \cup V_{k-1}$ and $v,w \in V_{k}$.
	We run Algorithm~\ref{algo:DFS} with input $F$, $k$ and  $V_1, \dots, V_{k}$.
	
	First, we prove that at any point in the execution of the algorithm, no edge is a transversal with respect to $T_1\cup\dots\cup T_k, S$.
	Suppose for a contradiction that at some point there is an edge which is a transversal.
	Note that $T_k\subseteq V_k$ and $S\subseteq V_k$.
	If $uv$ is this edge, then by the induction hypothesis and without loss of generality we have $u\in S$ and $v\in T_k$.
	This implies that when $u$ was moved from $U$ to $S$ (line~\ref{algo:m-1}), the set $U$ could have been extended, which means that $T_\ext\neq \emptyset$ and line~\ref{algo:m-1} would not have been executed, a contradiction.
	Now, assume $uv(w)$ is the transversal.
	Since $T_i\subseteq V_i$ for $i\in [k]$ and $S\subseteq V_k$, we have $w\in T_1\cup\dots\cup T_{k-1}$.
	Again, by the induction hypothesis and without loss of generality we have $u\in S$ and $v\in T_k$.
	Similarly as when we have an edge $uv$, at the time $u$ was moved from $U$ to $S$, the set $U$ could have been extended, a contradiction.
	
	Now we prove that at any point in the execution of the algorithm, there is no edge $uv(w)$ with $u \in T_1,\dots,T_{k-1},S$ and $v,w \in T_k$.
	Suppose for a contradiction that at some point there is such edge $uv(w)$.
	By the induction hypothesis we have $u\in S$ and $v,w\in T_k$, which again gives a contradiction as $U$ could have been extended.
	
	Note that since $F$ has at least $5^{k-1}cn$ vertices and no $(2,3)$-path with $n$ vertices, we have $|S|=cn$ at some point of the execution of Algorithm~\ref{algo:DFS}.
	Let $U$, $W_U$, $S$, $W_S$ and $T_1,\dots,T_k$ be the sets at that moment.
	Note that $|W_S|\leq |S|= cn$ and, since there is no $n$-vertex $(2,3)$-path in $F$, we have $|U|$, $|W_U|\leq n$.
	Therefore, $|T_i|\geq |V_i|-2cn\geq cn$ for $i\in [k-1]$ and $|T_k|\geq |V_k|-4cn \geq cn$.
	Put $V_i' = T_i$ for $i\in [k-1]$, $V_k'=S$ and $V_{k+1}'=T_k$.
	The sets $V_1',\dots,V_{k+1}'$ satisfies the requirements of the lemma. 
\end{proof}

\section{Proof of Theorem~\ref{thm:main}}
\label{sec:proof}

In this section we prove our main theorem.  We first define the
following constants:
\begin{align*}
\ell:=17, \quad k:=2 \ell,\quad \eps:=1/(k+1),\quad t:=8k+40k^2+5,
\quad\text{and}\quad t':=r_2(K_t^{(3)}), 
\end{align*}
where $r_s(K_t^{(3)}) := \min \{ n : K_n^{(3)} \rightarrow (K_t^{(3)})_s \}$
is the classical Ramsey-number for hypergraphs.
We start by obtaining a graph $G$ with bounded maximum degree and some
nice pseudorandom properties.  Let $a_{{\rm L}.\ref{lem:H2} }$ be
large enough to apply Lemma~\ref{lem:H2} with $k$ and $\eps$ and set
\begin{align*}
c := \eps a_{{\rm L}.\ref{lem:H2} }\ell.
\end{align*}
Let $a_{{\rm L}.\ref{lem:transversal}} = 5^{k}$ and note that $a_{{\rm L}.\ref{lem:transversal}}$ is large enough to apply Lemma~\ref{lem:transversal} with $k+1$, $c$ and $n$ and set
\begin{align*}
a :=2 a_{{\rm L}.\ref{lem:transversal}}.
\end{align*}

Lemma~\ref{lem:Hexists} applied with $\eps$ and $a$ provides a
constant~$b$.  Let $n$ be sufficiently large.  Then, from
Lemma~\ref{lem:Hexists} we know that there is a graph $G$ on $an$
vertices with maximum degree $b$ such that~\ref{prop:exist} holds.
Fix such a graph~$G$.

Now let $G^k(t')$ be the graph obtained from $G^k$ -- the $k$-th 
power of $G$ -- by replacing every vertex by a $K_{t'}$ and every edge by a
$K_{t',t'}$.  Finally, $H$ is the $3$-uniform hypergraph with vertex
set $V(G^k(t'))$ and a triple of vertices $xyz$ is an edge in $H$ if
and only if $xyz$ forms a triangle in $G^k(t')$.  For every
$v \in V(G)$ we denote by $H(v)$ the corresponding cluster consisting
of a $K^{(3)}_{t'}$ in $H$.  We claim that
$H \rightarrow (P^{(3)}_{n})_2$.  Since $|V(H)| = a t' n$ and
$\Delta(H) \le b^{2k+2} t'^3$, this would prove Theorem~\ref{thm:main}.

The rest of the proof is devoted to proving
that~$H \rightarrow (P^{(3)}_{n})_2$.  Fix a $2$-colouring of the
triples of~$H$. As $t' \ge r_2(K^{(3)}_t)$ for every $v \in V(G)$ 
the cluster $H(v)$ either contains a red or blue copy of $K^{(3)}_t$.
W.l.o.g.~there is a set of vertices $V \subseteq V(G)$ with
$|V| \ge an/2 = a_{{\rm L}.\ref{lem:transversal}} n$ such that for
all $v \in V$ the cluster $H(v)$ contains a blue copy of $K^{(3)}_t$, 
which we denote by $H'(v)$.
We let $H' \subseteq H$ be the $3$-graph induced by the clusters 
$H'(v)$ for $v \in V$.

We will define an auxiliary $(2,3)$-graph~$F$ on the vertex
set~$V$, whose edges will indicate that we can walk between the clusters
using blue triples of~$H'$.  Formally, for $u,v \in V$ a \emph{$(2,2)$-connector}
between the clusters $H'(u)$ and $H'(v)$ consists of four vertices
$x_1,x_2 \in H'(u)$ and $y_1,y_2 \in H'(v)$ such that $x_1x_2y_1$ and
$y_1y_2x_1$ are triples of~$H$.  Similarly, for $u,v,w \in V$ a
\emph{$(2,1,2)$-connector} between the clusters $H'(u)$ and $H'(v)$
through $H'(w)$ consists of five vertices $x_1,x_2 \in H'(u)$,
$z \in H'(w)$, and $y_1,y_2 \in H'(v)$ such that $x_1x_2z$, $x_1zy_1$,
and $zy_1y_2$ are triples of~$H$; see Figure~\ref{fig:connectors}.  We
then define a \emph{$(6,6)$-connector} (\emph{$(6,3,6)$-connector}) between 
$H'(u)$ and $H'(v)$ (through $H'(w)$) as the
disjoint union of three $(2,2)$-connectors
($(2,1,2)$-connectors) between $H'(u)$ and $H'(v)$ (through $H'(w)$).
\begin{figure}[htbp]
	\begin{center}
	\begin{tikzpicture}
	
	\draw [rounded corners=2mm,fill=gray,opacity=0.4] (-.3,-.3)--(-.3,1.8)--(.3,1.8)--(4.3,.2)--(4.3,-.3)--cycle;
	\draw [rounded corners=2mm,fill=gray,opacity=0.4] (4.3,-.3)--(4.3,1.8)--(3.7,1.8)--(-.3,.2)--(-.3,-.3)--cycle;
	
	\draw[thick] (2,2)--(2,-1);
	\node[draw,circle,inner sep=3pt,fill] at (0,0) {};
	\node[draw,circle,inner sep=3pt,fill] at (0,1.5) {};
	\node[draw,circle,inner sep=3pt,fill] at (4,0) {};
	\node[draw,circle,inner sep=3pt,fill] at (4,1.5) {};
	
	\node at (-.7,0) {$x_1$};
	\node at (-.7,1.5) {$x_2$};
	\node at (4.7,0) {$y_1$};
	\node at (4.7,1.5) {$y_2$};
	
	\node at (.5,-.7) {$H'(u)$};
	\node at (3.5,-.7) {$H'(v)$};
	
	\end{tikzpicture}
	\quad
	\begin{tikzpicture}
	
	\draw [rounded corners=2mm,fill=gray,opacity=0.4] (1.3,.3)--(1,.1)--(-.4,1.5)--(-.1,1.75)--(2,2.9)--(2.35,2.5)--cycle;
	\draw [rounded corners=2mm,fill=gray,opacity=0.4] (2.7,.3)--(3,.1)--(4.4,1.5)--(4.1,1.75)--(2,2.9)--(1.65,2.5)--cycle;
	\draw [rounded corners=2mm,fill=gray,opacity=0.4] (.6,.5)--(1,.1)--(3,.1)--(3.4,.5)--(2.2,2.8)--(1.8,2.8)--cycle;	
	
	\draw[thick] (2,1)--(2,-.5);
	\draw[thick] (0.5,2.5)--(2,1)--(3.5,2.5);
	\node[draw,circle,inner sep=3pt,fill] at (1,.5) {};
	\node[draw,circle,inner sep=3pt,fill] at (0,1.5) {};
	\node[draw,circle,inner sep=3pt,fill] at (3,.5) {};
	\node[draw,circle,inner sep=3pt,fill] at (4,1.5) {};
	\node[draw,circle,inner sep=3pt,fill] at (2,2.5) {};
	
	\node at (.7,0) {$x_1$};
	\node at (-.5,1.2) {$x_2$};
	\node at (3.3,0) {$y_1$};
	\node at (4.5,1.2) {$y_2$};
	\node at (2,3) {$z$};	
	
	\node at (-.4,.5) {$H'(u)$};
	\node at (4.4,.5) {$H'(v)$};
	\node at (2,3.5) {$H'(w)$};	
	
	\end{tikzpicture}
	\end{center}
	\caption{A $(2,2)$-connector and a $(2,1,2)$-connector.}
	\label{fig:connectors}
\end{figure}
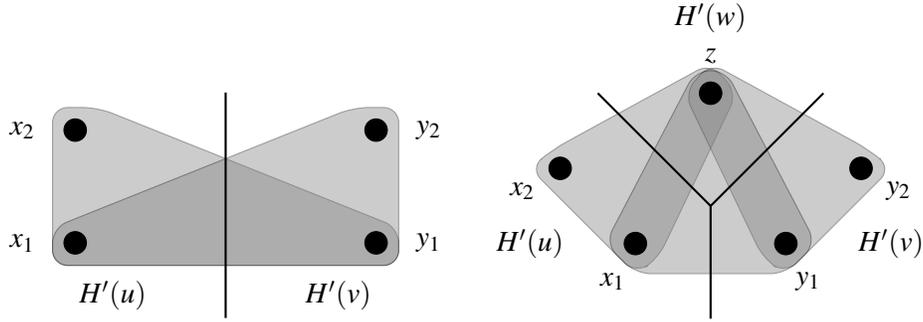

Let $F$ be a $(2,3)$-graph on the vertex set $V$ with the following
two types of edges:
\begin{enumerate}[label=\upshape({\itshape\roman*\,\/})]
	\item $uv\in E(F)$ if and only if there is a $(6,6)$-connector in blue
	between the corresponding clusters $H'(u)$ and $H'(v)$;
	\item $uv(w)\in E(F)$ if and only if there is a $(6,3,6)$-connector in
	blue between the corresponding clusters~$H'(u)$ and~$H'(v)$ through
	$H'(w)$.
\end{enumerate}

Suppose that $F$ contains a $(2,3)$-path on $n$ vertices
$v_1,\dots,v_n$ with the $w_i$ vertices all distinct for the $3$-edges.
We can turn this $(2,3)$-path into a blue tight path~$P^{(3)}_n$ in~$H$ as follows.
First, by following the $(2,3)$-path for $i = 1,\dots,n-1$, we choose a $(2,2)$-connector between $H'(v_i)$ and $H'(v_{i+1})$ if $v_iv_{i+1} \in E(F)$ or a $(2,1,2)$-connector between $H'(v_i)$ and $H'(v_{i+1})$ through $H'(w_i)$ if $v_iv_{i+1}(w_i) \in E(F)$ in such a way that they are all pairwise vertex-disjoint. 
This is possible, because we have $(6,6)$-connectors and $(6,3,6)$-connectors available and two vertices $y_1$ and $y_2$ from the previous connector can occupy at most two out of the three disjoint copies of connectors that are provided.
Then, for each of the clusters with $y_1$, $y_2$ and $x_1$, $x_2$ the vertices of the connectors within that cluster, we use the edges $y_1y_2x_2$ and $y_2x_2x_1$ to connect both connectors.
As the connectors only use blue edges and all edges within the clusters are blue this is a tight path only using blue edges.
Thus, in this case, we are able to obtain a blue~$P^{(3)}_n$ in~$H$ and we are done.

We assume that $F$ contains no $(2,3)$-path on $n$ vertices.
From Lemma~\ref{lem:transversal}, there exist pairwise disjoint sets 
$V_1,\dots,V_{k+1} \subseteq V$ of size at least $cn$ such that no
edge from $E(F)$ is a transversal.
We may assume that all these sets~$V_j$ have size exactly~$cn$.
Let $G'=G[V_1\cup\dots\cup V_{k+1}]$ and set $m := \ell n$.

We now want to find a path $P_m$ alternating through $V_1,\dots,V_{k+1}$
with edges in $G'\subseteq G$ using Lemma~\ref{lem:H2}. Since
$c=\eps a_{{\rm L}.\ref{lem:H2} }\ell$ and $\eps=1/(k+1)$, we have
$|V(G')| = (k+1) c n = a_{{\rm L}.\ref{lem:H2}}\ell n =a_{{\rm
		L}.\ref{lem:H2}}m$.  Also, we have
$|V_i| =  cn = \eps a_{{\rm L}.\ref{lem:H2}} m$ for $1\leq i\leq k+1$. 
As~$G'$ is an induced
subgraph of~$G$ and property~\ref{prop:exist} holds in $G$, property~\hyperref[prop:exist]{\upshape(P1$_m$)} does hold for~$G'$.
Therefore, by Lemma~\ref{lem:H2}, we conclude that there is a path
$P_m=P_{\ell n}$ with vertices alternating through $V_1,\dots,V_{k+1}$
and with edges in $G'\subseteq G$.

This path $P_{\ell n}$ gives us the $k$th power $P^k_{\ell n}$
in~$G^k$.  By the choice of $V_1,\dots,V_{k+1}$ no edge of
$P^k_{\ell n}$ is from $E(F)$ and also no triangle in $P^k_{\ell n}$
induces an edge $uv(w) \in E(F)$.  It remains to turn this
$P^k_{\ell n}$ into a red $P^{(3)}_n$ in~$H'$.

\begin{claim}
	\label{lem:redpath}
	If there is a $P_{\ell n}^k$ in $G^k$ that does not contain any
	edges from $F$, then there is a red~$P^{(3)}_n$ in $H'$.
\end{claim}

Let $P_{\ell n}^k=(v_1,\dots,v_{\ell n})$ and recall $H'(v_i)$ is the
cluster in $H'$ corresponding to the vertex $v_i$ for
$i=1,\dots,\ell n$.  We want to remove all vertices of $H'$ which
belong to blue $(2,2)$-connectors and $(2,1,2)$-connectors from
clusters along edges and triangles of $P_{\ell n}^k$.  
In $P_{\ell n}^k$ every vertex $v_i$ is incident to at most $2k$ other vertices in
$\{v_1,\dots,v_{\ell n}\}$ ($2k$ is the maximum degree of the $v_i$
in $P^k_{\ell n}$).  Also, every $v_i$ is contained in at most
$4k^2$ triangles of $P_{\ell n}^k$ together with two other vertices in
$\{v_1,\dots,v_{\ell n}\}$.

Let $u$ and $v$ be neighbours in $P_{\ell n}^k$.  Since there is no
blue $(6,6)$-connector between $H'(u)$ and $H'(v)$, there are at most
two $(2,2)$-connectors that do not overlap between $H'(u)$ and
$H'(v)$, which can both be deleted by removing at most $4$ vertices in each
cluster.  Let $u$, $v$, and $w$ be vertices that form a triangle in
$P_{\ell n}^k$.  Since there is no blue $(6,3,6)$-connector between
$H'(u)$, $H'(v)$, and $H'(w)$, there are at most six
$(2,1,2)$-connectors that do not overlap, two for each possibility to
place the single vertex.  These can be deleted by removing at most
$10$ vertices from each cluster.

By the above argument, we have to delete at most
$4(2k) + 10(4 k^2) \le t - 5$ vertices from every cluster to get rid
of all $(2,2)$-connectors and $(2,1,2)$-connectors.  Let
$H^*(v_i) \subseteq H'(v_i)$ be the remainder of the corresponding
cluster in $H'$, and note that $|H^*(v_i)| \ge 5$ for
$i=1,\dots,\ell n$.

A tuple $(u,v)$ is an \emph{end-tuple} of a tight path with at
least~$4$ vertices if~$u$ and~$v$ are consecutive vertices in the path
and~$u$ is contained in exactly two edges and $v$ is contained in
exactly one.  
The two tuples~$(u,v)$ and~$(v,w)$ are the end-tuples of the tight path~$(u,v,w)$ of
length~$3$.
Furthermore, every tuple $(u,v)$ is an end-tuple of the tight
path~$(u,v)$ of length $2$.

\begin{definition}
	For $i=1,\dots,n-1$ we say that the quadruple $(u_1,u_2,w_1,w_2)$ satisfies 
	property $Q_i$ if the following conditions hold:
	\begin{enumerate}
		\item $u_1,u_2,w_1,w_2$ are distinct vertices from $H'$ such that the pairs
		$u_1,u_2$ and $w_1,w_2$ are in clusters
		$H_r,H_s \in \{H^*(v_{(i-1)\ell +1}), \dots, H^*(v_{i\ell})\}$, respectively, where
		$r\neq s$;
		\item each of $(u_1,u_2)$ and $(w_1,w_2)$ is an end-tuple of a
		red tight path of length at least $i+1$ with vertices in
		$H^*(v_1)\cup \dots\cup H^*(v_{i\ell})$.
	\end{enumerate}
\end{definition}

To prove Claim~\ref{lem:redpath}, i.e., construct $P^{(3)}_n$ in red,
it is then sufficient to construct a quadruple satisfying property
$Q_{n-1}$.  We will construct this quadruple inductively.  The base case
$Q_1$ asks for two paths of length $2$ and, therefore, it is enough to
choose any pair $u_1,u_2$ from $H^*(v_{1})$ and $w_1,w_2$ from
$H^*(v_{2})$.  Therefore, the following is immediate:
\begin{equation}
\label{lem:basecase}
\text{There exists a quadruple $(u_1,u_2,w_1,w_2)$ for which $Q_1$
	holds.} 
\end{equation}

We will inductively find a quadruple $(u_1,u_2,w_1,w_2)$ satisfying
property $Q_i$ for every $i=2,\dots,n-1$.  Ultimately, after $n$ steps,
this gives us $P^{(3)}_n$ in red in $H'$.  Suppose $(u_1,u_2,w_1,w_2)$
satisfies property $Q_{i-1}$ for some $1<i\leq n-1$.  In the inductive
step, we obtain $(x_1,x_2,y_1,y_2)$ satisfying property $Q_i$ by
extending one of the paths ending in $(u_1,u_2)$ or $(w_1,w_2)$ to get
two longer paths ending in $(x_1,x_2)$ and $(y_1,y_2)$.  This mainly
relies on the absence of $(2,2)$-connectors and $(2,1,2)$-connectors
in blue and that there are $\ell$ clusters $H^*(v_{(i-1)\ell+1}),\dots,H^*(v_{i \ell})$ 
to choose $x_1,x_2,y_1,y_2$ from. As $k \ge 2\ell$, all edges are present between 
these clusters and the clusters $H_r$ and $H_s$ containing $u_1,u_2$ 
and $w_1,w_2$, respectively. 

\begin{fact}
	\label{lem:inductionstep}
	Let $1< i \le n-1$ and suppose $(u_1,u_2,w_1,w_2)$ satisfies
	property~$Q_{i-1}$.  Then there is a quadruple $(x_1,x_2,y_1,y_2)$
	that satisfies property~$Q_{i}$.
\end{fact}

\begin{proof}
	We will find~$(x_1,x_2,y_1,y_2)$ following the strategy sketched
	above.  Since there are no blue $(2,2)$-connectors between the
	clusters corresponding to $u_1,u_2$ and $w_1,w_2$, and  all
	possible triples between these clusters are edges in $H'$, then either the triple
	$u_1u_2w_2$ or the triple $w_1w_2u_2$ is red, say, w.l.o.g.,
	$u_1 u_2 w_2$ is red. We let the red path of length at least $i$ that ends in $(u_1,u_2)$ 
	be called $P_\red$ and note that the triple $u_1 u_2 w_2$ already 
	extends this path. We
	will show that it is possible to further extend this path to obtain two
	longer red tight paths with ends $(x_1,x_2)$ and
	$(y_1,y_2)$, respectively, such that $(x_1,x_2,y_1,y_2)$ satisfies property
	$Q_{i}$. 
	
	Notice that, as $\ell \ge 17$, by the pigeonhole principle there are
	nine sets $X_1,\dots,X_9$ each contained in a different cluster from
	$H^*(v_{(i-1)\ell +1}), \dots, H^*(v_{i\ell})$ and of size
	$|X_j| \ge 3$ for $j \in [9]$ (here we use that $|H^*(v_i)|\ge5$)
	such that either
	\begin{align}
	\label{eq:edgesXpresent}
	\text{for all $j \in [9]$ and every $x \in X_j$ the triple $u_2w_2x$ is red }
	\end{align}
	or
	\begin{align}
	\label{eq:edgesXnotpresent}
	\text{for all $j \in [9]$ and every $x \in X_j$ the triple $u_2w_2x$ is blue.}
	\end{align}
	
	We first consider the case where~\eqref{eq:edgesXpresent} holds.  It
	is enough to assume that we have $X_1,\dots,X_5$ and $|X_j|\ge 2$
	for $j \in [5]$.  If there are two sets $X,Y$ from $X_1,\dots,X_5$
	and $x_1,x_2 \in X$, $y_1,y_2 \in Y$ such that the triples $w_2x_1x_2$
	and $w_2y_1y_2$ are red, then the quadruple $(x_1,x_2,y_1,y_2)$
	satisfies property $Q_i$ as we can obtain two longer red paths by
	extending $P_\red$ following $u_1u_2w_2x_1x_2$ and
	$u_1u_2w_2y_1y_2$, respectively.  Otherwise, there are two sets
	$X,Y$ and $x_1,x_2 \in X$, $y_1,y_2 \in Y$ such that the triples
	$w_2x_1x_2$ and $w_2y_1y_2$ are blue.  As there is no blue
	$(2,1,2)$-connector, the triple $w_2x_1y_1$ is red.  There is no blue
	$(2,2)$-connector between the corresponding clusters, so either the
	triple $x_1y_1y_2$ or $y_1x_1x_2$ is red, say w.l.o.g.~$x_1y_1y_2$ is
	red.  This extends $P_\red$ by following $u_1u_2w_2x_1y_1y_2$ and
	gives the end-tuple $(y_1,y_2)$.  Repeating the same argument, which
	is possible, because there were five sets available (sets
	$X_1,\dots,X_5$), we get an end-tuple $(z_1,z_2)$ that extends
	$P_\red$ to a longer red path and, thus, a quadruple
	$(y_1,y_2,z_1,z_2)$ satisfying $Q_i$.
	
	In the case where~\eqref{eq:edgesXnotpresent} holds we proceed as
	follows.  As for all $j \in [9]$ there is no blue
	$(2,1,2)$-connector between the clusters of $u_1,u_2$ and $w_1,w_2$
	and $X_j$, we have for every $x \in X_j$ that either the triple
	$w_1w_2x$ or the triple $u_1u_2x$ is red.  Then we can assume by the
	pigeonhole principle w.l.o.g.~(we will not use the triple $u_1u_2w_1$)
	that there are sets $X_j' \subseteq X_j$ with $|X_j'|\ge 2$ for
	$j \in [5]$ (here we use that $|X_j|\ge 3$) such that
	\begin{align*}
	\text{for all $j \in [5]$ and every $x \in X_j'$ the triple $w_1w_2x$ is red.}
	\end{align*}
	Now we can continue exactly as in the case
	where~\eqref{eq:edgesXpresent} holds, with $u_1u_2$ replaced by
	$w_1w_2$ throughout and extending the path with end-tuple
	$(w_1,w_2)$.  Observing that the red tight paths that we have
	constructed have length at least $i+1$, we see that
	Fact~\ref{lem:inductionstep} is proved.
\end{proof}

Fact~\ref{lem:inductionstep} together with~\eqref{lem:basecase}
finishes the proof of Claim~\ref{lem:redpath} and hence the
proof of Theorem~\ref{thm:main} is complete.



\bibliographystyle{amsplain}

\begin{thebibliography}{99}
\bibitem{AlCh88}
N.~Alon and F.~R.~K. Chung, \emph{Explicit construction of linear sized
	tolerant networks}, Discrete Math. \textbf{72} (1988), no.~1-3, 15--19.

\bibitem{BalDeBiasio}
D.~Bal and L.~DeBiasio, \emph{New lower bounds on the size-Ramsey number of a path}, 2019, arXiv:1909.06354.

\bibitem{Be83}
J{\'o}zsef Beck, \emph{On size {R}amsey number of paths, trees, and circuits.
	{I}}, J. Graph Theory \textbf{7} (1983), no.~1, 115--129.

\bibitem{ben-eliezer12:_Ramsey}
Ido Ben-Eliezer, Michael Krivelevich, and Benny Sudakov, \emph{The size
	{R}amsey number of a directed path}, J. Combin. Theory Ser. B \textbf{102}
(2012), no.~3, 743--755.

\bibitem{BKMMMMP}
S.~Berger, Y.~Kohayakawa, G.~S. Maesaka, T.~Martins, W.~Mendon{\c c}a, G.~O.
Mota, and O.~Parczyk, \emph{The size-{R}amsey number of powers of bounded
	degree trees}, J. Lond. Math. Soc. (2020), 1--19.

\bibitem{bollobas1986extremal}
B\'{e}la Bollob\'{a}s, \emph{Extremal graph theory with emphasis on
	probabilistic methods}, CBMS Regional Conference Series in Mathematics,
vol.~62, Published for the Conference Board of the Mathematical Sciences,
Washington, DC; by the American Mathematical Society, Providence, RI, 1986.

\bibitem{2colourSizeRamsey}
D.~Clemens, M.~Jenssen, Y.~Kohayakawa, N.~Morrison, G.~O. Mota, D.~Reding, and
B.~Roberts, \emph{The size-{R}amsey number of powers of paths}, J. Graph
Theory \textbf{91} (2019), no.~3, 290--299.

\bibitem{DFMR_sizeRamsey}
Andrzej Dudek, Steven~La Fleur, Dhruv Mubayi, and Vojtech R{\"o}dl, \emph{On
	the size-ramsey number of hypergraphs}, Journal of Graph Theory \textbf{86}
(2017), no.~1, 104--121.

\bibitem{DuPr15}
Andrzej Dudek and Pawe{\l} Pra{\l}at, \emph{An alternative proof of the linearity
	of the size-{R}amsey number of paths}, Combin. Probab. Comput. \textbf{24}
(2015), no.~3, 551--555.

\bibitem{dudek2017some}
\bysame, \emph{On some multicolor {R}amsey properties of random graphs}, SIAM
J. Discrete Math. \textbf{31} (2017), no.~3, 2079--2092.

\bibitem{erdHos1981combinatorial}
P.~Erd{\H{o}}s, \emph{On the combinatorial problems which {I} would most like
	to see solved}, Combinatorica \textbf{1} (1981), no.~1, 25--42.

\bibitem{CYKMMRB}
J.~Han, M.~Jenssen, Y.~Kohayakawa, G.~O. Mota, and B.~Roberts, \emph{The
	multicolour size-{R}amsey number of powers of paths}, J. Combin. Theory Ser. B
\textbf{145} (2020), 359--375.

\bibitem{krivelevich2017long}
M.~Krivelevich, \emph{Long cycles in locally expanding graphs, with
	applications}, Combinatorica (2018). To
appear.

\bibitem{letzter16:_path_ramsey}
Shoham Letzter, \emph{Path {R}amsey number for random graphs}, Combin. Probab.
Comput. \textbf{25} (2016), no.~4, 612--622.

\bibitem{lu2018size}
Linyuan Lu and Zhiyu Wang, \emph{On the size-ramsey number of tight paths},
SIAM Journal on Discrete Mathematics \textbf{32} (2018), no.~3, 2172--2179.

\end{thebibliography}
\providecommand{\bysame}{\leavevmode\hbox to3em{\hrulefill}\thinspace}


\begin{aicauthors}
\begin{authorinfo}[jie]
  Jie Han\\
  School of Mathematics and Statistics\\
  Beijing Institute of Technology\\
  Beijing, China\\
  jie\_han\imageat{}uri\imagedot{}edu
\end{authorinfo}
\begin{authorinfo}[yoshi]
  Yoshiharu Kohayakawa\\
  Instituto de Matem\'atica e Estat\'{\i}stica\\
  Universidade de S\~ao Paulo\\
  Rua do Mat\~ao 1010, 05508-090 S\~ao Paulo, Brazil\\
  yoshi\imageat{}ime\imagedot{}usp\imagedot{}br
\end{authorinfo}
\begin{authorinfo}[shoham]
  Shoham Letzter\\
  Department of Mathematics\\
  University College London\\
  Gower Street, London WC1E 6BT\\
  s.letzter\imageat{}ucl\imagedot{}ac\imagedot{}uk
\end{authorinfo}
\begin{authorinfo}[gui]
  Guilherme Oliveira Mota \\
  Instituto de Matem\'atica e Estat\'{\i}stica\\
  Universidade de S\~ao Paulo\\
  Rua do Mat\~ao 1010, 05508-090 S\~ao Paulo, Brazil\\
  mota\imageat{}ime\imagedot{}usp\imagedot{}br
\end{authorinfo}
\begin{authorinfo}[olaf]
	Olaf Parczyk \\
	London School of Economics\\
	Department of Mathematics\\
	Houghton Street, London, WC2A 2AE, UK\\
	o\imagedot{}parczyk\imageat{}lse\imagedot{}ac\imagedot{}uk
\end{authorinfo}
\end{aicauthors}

\end{document}